\documentclass[11pt]{amsart}

\usepackage{amsfonts,amsmath,amssymb,amscd,amsthm}
\usepackage[alphabetic]{amsrefs}

\newtheorem{lemma}{Lemma}[section]
\newtheorem{theorem}[lemma]{Theorem}
\newtheorem{corollary}[lemma]{Corollary}
\newtheorem{proposition}[lemma]{Proposition}

\theoremstyle{definition}
\newtheorem{definition}[lemma]{Definition}
\newtheorem{notation}[lemma]{Notation}
\newtheorem{question}[lemma]{Question}
\newtheorem*{question*}{Question}

\theoremstyle{remark}

\newtheorem{example}[lemma]{Example}

\newcommand{\TA}{\textup{TA}}
\newcommand{\GA}{\textup{GA}}
\newcommand{\BA}{\textup{BA}}

\newcommand{\Aff}{\textup{Aff}}

\newcommand{\sdeg}{\textup{sdeg}}
\newcommand{\Spec}{\textup{Spec}}

\newcommand{\C}{\mathbb{C}}

\newcommand{\Q}{\mathbb{Q}}
\newcommand{\A}{\mathbb{A}}

\newcommand{\GG}{\mathcal{G}}
\newcommand{\TT}{\mathcal{T}}
\newcommand{\LL}{\mathcal{L}}
\newcommand{\EE}{\mathcal{E}}
\newcommand{\BB}{\mathcal{B}}

\title[On the closure of the tame automorphism group]{On the closure of the tame automorphism group of affine three-space}
\subjclass[2010]{14R10}
\keywords{polynomial automorphisms, tame and wild automorphisms, ind-varieties}

\thanks{This research was supported in part by the ANR Grant BirPol ANR-11-JS01-004-01}

\author{\'Eric Edo}
\address{Universit\'e de la Nouvelle Cal\'edonie, Bat.~S, Campus de Nouville, BP R4 -- 98851 Noum\'ea CEDEX, Nouvelle Cal\'edonie.}
\email{eric.edo@univ-nc.nc}

\author{Pierre-Marie Poloni}
\address{Universit\"{a}t Basel, Mathematisches Institut, Rheinsprung $21$, CH-$4051$ Basel, Switzerland.}
\email{pierre-marie.poloni@unibas.ch}

\begin{document}

\begin{abstract} We provide explicit families of tame automorphisms of the complex affine three-space which degenerate to wild automorphisms. This shows that the tame subgroup of the group of polynomial automorphisms of $\C^3$ is not closed, when the latter is seen as an infinite dimensional algebraic group.
\end{abstract}

\maketitle

\section{Introduction}

In 1965 \cite{Sha66}, Shafarevich introduced the notions of infinite-dimensional varieties and infinite-dimensional algebraic groups, now usually called ind-varieties and ind-groups. His main motivation was to study the group $\GA_n(\C)$ of polynomial automorphisms of complex affine $n$-spaces $\A_{\C}^n=\C^n$, which he endowed with an ind-group structure. This new approach was fruitful, since it allows him to state many nice (and tempting) results in \cites{Sha66,Sha81, Sha95}. In the present paper, we are interested in one that claims
that the tame automorphisms form a dense subgroup $\TA_n(\C)$ of  $\GA_n(\C)$. Unfortunately, Shafarevich's proof  is based on another result -- namely that a closed subgroup $H$ of a connected ind-group $G$ with the same Lie algebra as $G$ is equal to $G$ --, which turns out
to be false, since Furter and Kraft constructed recently a counterexample to
that  statement in \cite{FK}. Therefore, we must reconsider the question of the density of the tame subgroup and ask it again.
\begin{question*} Is $\TA_n(\C)$ dense in $\GA_n(\C)$ in the topology of ind-varieties (for $n\geq3$)?
\end{question*}

Moreover,  Furter and Kraft also establish the following surprising result: the subgroup $\TA_2(\C[z])$ of tame automorphisms of $\C^3$ that fix the last coordinate is a closed subgroup of the group $\GA_2(\C[z])$ of polynomial automorphisms of $\C^3$ that fix the last coordinate. In light of this, and since there were no known examples of  non-tame
automorphism that belong to the closure of the tame subgroup, we may
even ask if the tame subgroup is closed in $\GA_n(\C)$.

\begin{question*} Is $\TA_n(\C)$ closed in $\GA_n(\C)$ in the topology of ind-varieties (for $n\geq3$)?
\end{question*}

We will show that it is not the case, when $n=3$ of course, since it is the only case where the existence of non-tame automorphisms is proved. We  indeed construct families of tame automorphisms of $\C^3$ of bounded degrees which degenerate to wild (i.e.~non-tame) automorphisms. In particular, we will prove that the automorphism $\varphi$ of $\C^3$ defined by
$$\varphi= \left(x+\frac{3}{4}z^2y(\frac{3}{2}y^2-4xz)+\frac{3}{8}z^5(\frac{3}{2}y^2-4xz)^2, y+z^3(\frac{3}{2}y^2-4xz), z\right)$$
is not tame but is in the closure of the tame subgroup of $\C^3$. More precisely,  $\varphi$ is the limit, when $t\to0$, of the tame automorphism $\sigma_t$ of $\A^3_{\C(t)}$, which is given by  $$\sigma_t=(x-\frac{3yz}{2t}+\frac{z^3}{2t^2},y-\frac{z^2}{t},z)\circ(x,y,z+t^3x^2-t^2y^{3})\circ(x+\frac{3yz}{2t}+\frac{z^3}{t^2},y+\frac{z^2}{t},z)$$
and which has all its coefficient in $\C[t]$.

This example  illustrates a new phenomenon concerning the \emph{length} of tame automorphisms. Recall that the length of a (tame) automorphism $\sigma$ of $\C^n$ is the minimum number of triangular automorphisms that occur in a writing of $\sigma$ as  composition of affine and triangular automorphisms. Recall also that Furter proved in~\cite{FurterLength} that the length of automorphisms of the affine plane is lower semicontinuous.
That means that an automorphism of $\C^2$ of length $l$ can not be obtained as the limit of automorphisms of length $<l$. In contrast, in the above example, $\varphi$ is an automorphism of $\C^3$ of infinite length (because it is non-tame), which is the limit of the family $\sigma_t$ of automorphisms of length 3.

\medskip

The article is organized as follow. We fix some notations and recall the definition of tame automorphisms in Section \ref{sec-notations}. In Section~\ref{sec-limits}, we recall why the group $\GA_n(\C)$ of polynomial automorphisms of $\C^n$ has the structure of an infinite-dimensional affine algebraic variety and study the subset of $\GA_3(\C)$ which consists of all ``limits'' of tame automorphisms. In particular, we show that the set of tame automorphisms of $\C^3$ of degree at most $d$ is a constructible set in $\GA_3(\C)$ for all $d\geq1$. Finally, we give concrete examples of non-tame automorphisms which belong to the closure of the tame subgroup in Section~\ref{sec-examples}.

\section{Tame automorphisms}\label{sec-notations}

Let $n\ge 1$ be an integer and $R$ be a   commutative algebra over a field $k$. We denote by $R^{[n]}$ the polynomial algebra in $n$ variables over $R$. A polynomial map of $\A^n_R=\A^n_k\times_{\Spec(k)}\Spec(R)$ is a map  $f$ from $\A^n_R$ to itself of the form $$f:(x_1,\ldots,x_n)\mapsto (f_1(x_1,\ldots,x_n),\ldots, f_n(x_1,\ldots,x_n))$$
where the $f_i$'s belong to the polynomial ring $R[x_1,\ldots,x_n]$. We will denote by $f=(f_1,\ldots,f_n)$ such a map and we define its degree by
$$\deg(f)=\max\{\deg(f_i)\mid i=1,\ldots,n\}.$$
We will also denote by $f^{\ast}$ the corresponding $R$-algebra endomorphism of $R[x_1,\ldots,x_n]$, which is given by $f^{\ast}(P(x_1\ldots,x_n))=P(f_1,\ldots,f_n)$ for every element $P\in R[x_1,\ldots,x_n]$. Notice in particular that
$f^{\ast}(x_i)=f_i$ for all $i$.

The composition of two polynomial maps $f$ and $g$ is simply defined by $g\circ f=(g_1(f_1,\ldots,f_n),\ldots,g_n(f_1,\ldots,f_n))$.  We denote by $\GA_n(R)$ the group of (algebraic) automorphisms of $\A^n_R$ over $\Spec(R)$. An element $f\in\GA_n(R)$
 is simply an invertible polynomial map from $\A^n_R$ to $\A^n_R$ whose inverse $f^{-1}$ is also a polynomial map. We denote by $$\Aff_n(R)=\{f\in\GA_n(R)\mid \deg(f)=1\}$$
the \textit{affine} subgroup of $\GA_n(R)$ and by
$$\BA_n(R)=\{(f_1,\ldots,f_n)\in\GA_n(R)\mid f_i\in R[x_i,\ldots,x_n]\textrm{ for all }i=1\ldots n\}$$
the subgroup of \emph{triangular} automorphisms.

The subgroup of \textit{tame} automorphisms of $\A^n_R$ is denoted by $\TA_n(R)$. It is the subgroup of $\GA_n(R)$ generated by $\Aff_n(R)$ and $\BA_n(R)$. An element of $\GA_n(R)\smallsetminus\TA_n(R)$ is called \emph{wild}. The Tame Generators Problem asks for the existence of such automorphisms in the case where $R=k$ is a field.
\begin{question}[Tame Generators Problem]\label{que:tgp}\rm
Does it hold that ${\rm GA}_n(k)=\TA_n(k)$?
\end{question}
When $n=1$, the answer is trivially yes.
When $n=2$, the answer is also positive by the famous Jung-van der Kulk's theorem (cf.~\cite{Jung,Kulk}), which asserts  moreover that, for any field $k$, the group $\GA_2(k)=\TA_2(k)$ is the  amalgamated free product of $\Aff_2(k)$ and $\BA_2(k)$ along their intersection.

Note that we can consider $\GA_2(k[z])$ as a subgroup of $\GA_3(k)$ via the map  sending an element $(f_1,f_2)\in\GA_2(k[z])$ onto the corresponding automorphism $(f_1,f_2,z)$ in $\GA_3(k)$, which fixes the last coordinate. Then, Shestakov and Umirbaev \cites{SU1,SU2} proved the following impressive result.
\begin{theorem}[Shestakov, Umirbaev, 2004]\label{thm:SU}
Let $k$ be a field of characteristic zero. Viewing $\GA_2(k[z])$ and $\TA_2(k[z])$ as subgroups of $\GA_3(k)$, we have
$$\GA_2(k[z])\cap\TA_3(k)={\rm TA}_2(k[z]).$$
\end{theorem}

This answers negatively the Tame Generators Problem in dimension $3$ in characteristic zero. For example, the famous Nagata automorphism $$(x-2y(y^2+zx)-z(y^{2}+zx)^2,y+z(y^2+zx),z)$$ is a wild automorphism of $\C^3$. Indeed, the amalgamated free product structure on $\GA_2(\C(z))$ gives an algorithm  to check if a given element in $\GA_2(\C[z])$ is in
$\TA_2(\C[z])$ or not (see \cite{JPF}), and it turns out that the first two components of Nagata's automorphism correspond to an automorphism in $\GA_2(\C[z])\smallsetminus\TA_2(\C[z])$. When $n\geq 4$, Question~\ref{que:tgp} is still open.

\section{Limits of tame automorphisms}\label{sec-limits}

We will now recall how one can see $\GA_n(\C)$ as an ind-variety (i.e.~infinite dimensional algebraic variety) and which topology we consider on it. To this purpose, it is convenient to work with a fixed integer $n\geq 2$ (for us, it will be $n=3$) and to let $\GG=\GA_n(\C)$. Then, we consider the filtration on $\GG$ given by the degree and we define for every $d\geq1$ the subset
$$\GG_{\leq d}=\{f\in\GG=\GA_n(\C)\mid \deg(f)\leq d\}$$
of polynomial automorphisms of degree at most $d$. More generally, a subset $S\subset \GG$ be given, we let
$$S_{\leq d}:=S\cap \GG_{\leq d}=\{f\in S \mid \deg(f)\leq d\}.$$

One can show (see e.g.~\cite{Furter}) that each $\GG_{\leq d}$ has the structure of an affine algebraic variety and  is closed in $\GG_{\leq d+1}$ in the Zariski topology. Therefore, $\GA_n(\C)=\GG=\bigcup_{d \geq 1}\GG_{\leq d}$ is an ind-variety in the sense of Shafarevich. As usual, we endow it with the ind-topology in which a subset $S\subset\GG$ is closed if and only if every subset $S_{\leq d}$ is closed in $\GG_{\leq d}$ in the Zariski topology.

In the present paper, we focus our attention on the subset of tame automorphisms of affine three-space, which we will denote by $\TT$ in the sequel. We would like to investigate which automorphisms of $\C^3$ can be obtain as limits of such tame automorphisms in the following sense.

\begin{definition}
Let $S\subset\GG$ be a subset and let $f\in\GG$. We say that $f$ is \emph{limit} of elements of $S$ if there exist a positive integer $d\geq1$ and a subset $U\subset S_{\leq d}$ which is locally closed in $\GG_{\leq d}$   such that $f\in\overline{U}$.

Following \cite{FK}, we say that a subset $S\subset\GG$ is \emph{weakly closed} if $S$ contains all limits of elements of $S$.
\end{definition}

Together with a valuative criterion due to Furter \cite{Furter}, Theorem~\ref{thm-SU-decomposition} allows us to make this definition of limit more concrete in the case of tame automorphisms of $\C^3$ (see Corollary~\ref{cor-valuatif} below). This result follows  from the theory of Shestakov and Umirbaev, which provides an algorithm to decompose every tame automorphism of $\C^3$ as a product of affine and triangular maps.

\begin{theorem}\label{thm-SU-decomposition}
For each $d\geq 1$, there exist positive integers $m=m(d)$ and $k=k(d)$, depending only on $d$, such that every tame automorphism $f\in\TT_{\leq d}$ of $\C^3$ of degree at most $d$ can be written as a composition $f=f_1\circ\ldots\circ f_m$, where each $f_i$ is either affine, or a triangular automorphism of $\C^3$ of degree at most $k$.
\end{theorem}

\begin{proof}
The proof is based on Shestakov-Umirbaev theory of reduction of tame automorphisms of $\C^3$. These reductions involve affine and elementary automorphisms. Let us recall that an automorphism $f$ of $\C^n$ is called \emph{elementary automorphism} if it  is of the form $$f=(x_1,\ldots,x_{i-1},x_i+P,x_{i+1},\ldots,x_n)$$ for some $1\leq i\leq n$ and some polynomial $P\in\C[x_1,\ldots,\widehat{x_i},\ldots,x_n]$ which does not depend on the variable $x_i$. We denote by $\EE$ the set of elementary automorphisms of $\C^3$ and by $\EE_{\leq d}$ the subset of those which are of degree at most $d$. Remark that a triangular automorphism of $\C^3$ of degree $d$ is equal to the composition of one affine automorphism and two elements of $\EE_{\leq d}$.

In \cite{SU2} Shestakov and Umirbaev did not consider the usual degree for a polynomial automorphism $f$, but the one that is given by the sum of the degrees of the components of $f$. Let us denote it by $\sdeg(f)$. So, we let $\sdeg((f_1,f_2,f_3))=\deg(f_1)+\deg(f_2)+\deg(f_3)$ for all automorphism $f=(f_1,f_2,f_3)$ of $\C^3$. According to the theory developed in \cite{SU2}, for every non-affine tame automorphism $f$ of $\C^3$, one of the following properties holds:
\begin{enumerate}
\item $f$ admits an elementary reduction, i.e.~there exists an elementary automorphism $e\in\EE$ such that $\sdeg(e\circ f)<\sdeg(f)$.
\item $f$ admits a reduction of Type I or II and then, there exist  an elementary automorphism $e\in\EE$ and  an affine one $a\in\GG_{\leq 1}$ such that $\sdeg(e\circ a\circ f)<\sdeg(f)$.
\item $f$ admits a reduction of Type III and then, there exist  an elementary automorphism $e\in\EE$, an elementary automorphism $e_2\in\EE_{\leq 2}$ of degree 2 and  an affine automorphism $a\in\GG_{\leq 1}$ such that $\sdeg(e\circ e_2\circ a\circ f)<\sdeg(f)$.
\end{enumerate}
Actually, Shestakov and Umirbaev also considered another kind of reduction, called of Type IV. But Kuroda proved that this one never occurs (see \cite{Kuroda}). Thus, any tame automorphism $f\in\TT_{\leq d}$ can be reduced to an affine automorphism by using at most $3d-3$ reductions as above. So, to conclude the proof, it only remains to show that all elementary maps that appear in a decomposition of an element of $\TT_{\leq d}$ can be taken with degrees bounded by a number $k=k(d)$ which depends only on $d$.

Let us first consider elementary reductions. Let $f=(f_1,f_2,f_3)\in\TT_{\leq d}$ and suppose that there exists a polynomial $P\in\C[x,y]$ such that $\deg(f_3-P(f_1,f_2))<\deg(f_3)$. We denote by $\overline{a}$ the homogeneous part of highest degree of a polynomial $a\in\C[x_1,x_2,x_3]$.

If $\overline{f_1}$ and $\overline{f_2}$ are algebraically independent, then the equality $\overline{f_3}=\overline{P(f_1,f_2)}=\overline{P}(\overline{f_1},\overline{f_2})$ easily implies that $\deg(P)\leq\deg(f_3)\leq d$. So, let us assume that  $\overline{f_1}$ and $\overline{f_2}$ are algebraically dependent. Now, if $\overline{f_1}\in\C[\overline{f_2}]$, then $\overline{f_1}=\lambda\overline{f_2}^{\alpha}$ for some $\lambda\in\C^*$ and $1\leq \alpha\leq d$. In this case, instead of the elementary reduction $e\circ f$ of $f$ given by $e=(x_1,x_2,x_3-P(x_1,x_2))$, we can perform another elementary reduction, namely the one given $\widetilde{e}=(x_1-\lambda x_2^{\alpha},x_2,x_3)$.

Therefore, we can suppose that $\overline{f_1}$ and $\overline{f_2}$ are algebraically dependent and that $\overline{f_1}\notin\C[\overline{f_2}]$ and $\overline{f_2}\notin\C[\overline{f_1}]$. A pair $(f_1,f_2)$ of such polynomials is called $*$\emph{-reduced} in \cites{SU1,SU2}.  Without loss of generality, we can assume that $d_1=\deg(f_1)<d_2=\deg(f_2)\leq d$. Following \cite{SU1}, we also let $p=d_1\gcd(d_1,d_2)^{-1}$, $s=d_2\gcd(d_1,d_2)^{-1}$, $\deg_{x_2}(P(x_1,x_2))=pq+r$ and $\deg_{x_1}(P(x_1,x_2))=sq_1+r_1$, where $0\leq r< p$ and $0\leq r_1< s$. Then, Theorem 3 in \cite{SU1} gives the following inequalities.
$$\deg(P(f_1,f_2))\geq qN+d_2r\geq q \quad \text{and}\quad \deg(P(f_1,f_2))\geq q_1N+d_1r_1\geq q_1, $$
where $N:=d_1d_2\gcd(d_1,d_2)^{-1}-d_1-d_2+\deg[f_1,f_2]\geq2$.
Since $\deg(P(f_1,f_2))=\deg(f_3)\leq d$, we  obtain
$$\deg_{x_2}(P(x_1,x_2))=pq+r\leq d_1d+d_1\leq d(d+1)$$
and
$$\deg_{x_1}(P(x_1,x_2))=sq_1+r_1\leq d_2d+d_2\leq d(d+1).$$
Thus, $\deg(P)\leq 2d(d+1)$ is bounded by a constant depending only on $d$, as desired.

Finally, it remains to consider the case where  $f$ admits a reduction of Type I, II or III. Let us write $\widetilde{f}=(g_1,g_2,f_3):=a\circ f$ if f admits a reduction of Type I or II and  $\widetilde{f}=(g_1,g_2,f_3):=e_2\circ a\circ f$ if f admits a reduction of Type III.  Furthermore, it follows from the precise definitions in \cite{SU2} that $(g_1,g_2)$ is a $*$-reduced pair and that $\widetilde{f}$ admits an elementary reduction $e\circ \widetilde{f}$ such that $\sdeg(e\circ \widetilde{f})<\sdeg(f)$. By the previous discussion, we can conclude that $\deg(e)\leq4d(2d+1)$. This proves the theorem.
\end{proof}

Theorem \ref{thm-SU-decomposition} implies the following fact.

\begin{proposition}\label{prop-constructible}
The set $\TT_{\leq d}$ of tame automorphisms of $\C^3$ of degree at most $d$ is  a constructible subset (i.e.~a finite union of locally closed subsets) of $\GG$ for every $d\geq 1$.
\end{proposition}

\begin{proof}Let $d\geq1$ be fixed, and let $m$ and $k$ be the corresponding integers given by Theorem \ref{thm-SU-decomposition}.
Denote by $\BB_{\leq k}$ the set of triangular automorphisms of $\C^3$ of degree at most $k$. Remark that this set is closed in $\GG_{\leq k}$ and thus in $\GG$. Consequently, the set $S:=\Aff_3(\C)\cup\BB_{\leq k}$ is  closed in $\GG$. The proposition  follows then by Chevalley's theorem. Indeed, the composition-map $\gamma_m:(\GG_{\leq k})^m\to\GG_{\leq k^m}$  defined by $\gamma_m(f_1,f_2,\ldots,f_m)=f_1\circ f_2\circ\cdots\circ f_m$ is a morphism of algebraic varieties. Therefore, $\gamma_m(S^m)$ is a constructible set in  $\GG_{\leq k^m}$. Hence,  $\TT_{\leq d}=\gamma_m(S^m)\cap\GG_{\leq d}$ is  constructible too.
\end{proof}

As a corollary, we obtain the following description of the set $\LL$ of all limits of tame automorphisms of $\C^3$.

\begin{corollary}\label{cor-valuatif}
The set $\LL:=\bigcup_{d\geq 1}\overline{\TT_{\leq d}}$ is weakly closed.

Furthermore, an automorphism $f\in\GG$ belongs to $\LL$ if and only if there exists a family $\varphi_t=((\varphi_t)_1,(\varphi_t)_2,(\varphi_t)_3)$ indexed by $t\in\C$ such that
\begin{enumerate}
\item $(\varphi_t)_i\in\C[[t]][x_1,x_2,x_3]$ for all $1\leq i\leq 3$;
\item $\varphi_t$ defines a tame automorphism of $\A_{\C((t))}^3$;
\item $\varphi_0=f$.
\end{enumerate}
\end{corollary}

\begin{proof}
By Proposition~\ref{prop-constructible}, we can write every set $\TT_{\leq k}$ as a finite union $\TT_{\leq k}=\bigcup_{i=1}^{n(k)}U_{k,i}$ of locally closed sets. In particular, the equality $\overline{\TT_{\leq k}}=\bigcup_{i=1}^{n(k)}\overline{U_{k,i}}$ holds for every $k\geq 1$. Now, let $U\subset \LL_{\leq d}$ be locally closed in  $\GG_{\leq d}$. Remark that
$$U=\bigcup_{k\geq 1}\overline{\TT_{\leq k}}\cap U=\bigcup_{k\geq 1}\left(\bigcup_{i=1}^{n(k)}\overline{U_{k,i}}\cap U\right)=\bigcup_{k\geq 1}S_k$$
where all $S_k:=\bigcup_{i=1}^{n(k)}\overline{U_{k,i}}\cap U$ are constructible subsets of $\GG_{\leq d}$. This implies (see e.g.~Lemma 2.5.4 in \cite{FK}) that there exists a $k_0\geq 1$ such that $U=\bigcup_{k=1}^{k_0}S_k$. Thus, $U\subset\overline{\TT_{\leq k_0}}$ and so $\overline{U}\subset\overline{\TT_{\leq k_0}}\subset\LL$. This proves that $\LL$ is weakly closed.

The second assertion of the corollary is a direct application of a valuative criterion due to Furter \cite{Furter}.
\end{proof}

To sum up, we have the three following inclusions involving $\TT$, $\LL$ and $\GG$.
$$\TT\subset\LL\subset\overline{\TT}\subset\GG.$$
The next section  is devoted to the proof that the first inclusion  is a strict one, i.e.~$\TT\varsubsetneqq\LL$. Since we do not know whether the two others are strict or not, we would like to ask two natural questions.

\begin{question}
Is $\LL$ closed in $\GG$? In other words, does the equality $\overline{\TT}=\LL$ hold?
\end{question}

\begin{question}
Is $\TT$ dense in $\GG$?
\end{question}

Note that these two questions are independent. Of course, $\LL$ could be closed and  not equal to the whole $\GG$. But, more surprisingly, even if $\TT$ would be dense in $\GG$, there may be some automorphisms of $\C^3$ which do not belong to $\LL$, i.e~ which we can not obtain as limits of tame automorphisms of bounded degree.

Finally, it is worth mentioning that $\LL$ and $\overline{\TT}$ are  subgroups of $\GG$.  Indeed, let us  recall that, as shown by Shafarevich in \cite{Sha81}, $\GG$ is  an infinite-dimensional algebraic group (ind-group for short). That means that the multiplication map $\mu:\GG\times\GG\to\GG$ and the inverse map $\iota:\GG\to\GG$ are morphisms of ind-varieties,  where a map $\varphi:X=\bigcup_{d}X_{\leq d}\to Y=\bigcup_{d}Y_{\leq d}$ between two ind-varieties is called \emph{morphism}, if for any $n\geq1$ there is an $m\geq1$ such that $\varphi(X_{\leq n})\subset Y_{\leq m}$ and $\varphi|_{X_{\leq n}}\to Y_{\leq m}$ is a morphism of algebraic varieties. With exactly the same arguments as for the classical case of algebraic groups (see e.g. Section 7.4 in~\cite{humphreys}), one checks that the following result remains true in the context of ind-groups.

\begin{lemma}\label{lemma-ind-subgroups}
Let $G$ be a ind-group and let $H\subset G$ be a subgroup. Then, the closure $\overline{H}$ of $H$ is a subgroup of $G$.
\end{lemma}

\begin{corollary}
The sets $\LL$ and $\overline{\TT}$ are subgroups of $\GA_3(\C)$.
\end{corollary}

\begin{proof}
By Lemma~\ref{lemma-ind-subgroups}, $\overline{\TT}$ is a  subgroup of $\GG$. On one hand, it follows from  Corollary~\ref{cor-valuatif} that $\LL$ is closed under composition. On the other hand, by a result due to Gabber (see Corollary 1.4 in~\cite{BCW}), we have $\iota(\TT_{\leq d})\subset \TT_{\leq d^2}$ for all $d\geq1$, where $\iota:\GG\to\GG$ denotes the inversion map. Therefore, $\iota(\overline{\TT_{\leq d}})\subset\overline{\TT_{\leq d^2}}$ for all $d\geq1$. This is a consequence of the fact that the inversion map is a morphism of ind-groups and of the following elementary topological argument: if $f:X\to Y$ is a continuous map between two topological spaces and if $f(A)\subset B$ for some subsets $A\subset X$ and $B\subset Y$, then  $f(\overline{A})\subset\overline{B}$. Hence, $\LL$ is closed under inversion.
\end{proof}

\section{Examples of wild limits}\label{sec-examples}

\begin{notation}
Let $n,m\geq1$ be positive integers. We let $\Delta\in\C[x,y,z]$ be given by $\Delta=zx+y^{m+1}$ and we consider the derivation $\delta$ of $\C[x,y,z]$ defined by
$$\delta=\Delta\left(z^n\frac{\partial}{\partial y}-(m+1)y^mz^{n-1}\frac{\partial}{\partial x}\right).$$
\end{notation}

Since $\Delta$ is in the kernel of the triangular derivation $z^n\frac{\partial}{\partial y}-(m+1)y^mz^{n-1}\frac{\partial}{\partial x}$, $\delta$ is a locally nilpotent derivation of $\C[x,y,z]$. Therefore, the exponential map  associated to $\delta$ is a polynomial automorphism of $\C^3$. Actually, it is an element of $\GA_2(\C[z])$ and we have
$$\varphi_{\lambda}=\exp(\lambda\delta)=(x-\sum_{k=1}^{m+1}{m+1\choose k}\lambda^k\Delta^ky^{m+1-k}z^{nk-1},y+\lambda\Delta z^n,z)$$
for all $\lambda\in\C$. Moreover, it follows from the theory developed by Shestakov and Umirbaev, and improved by Kuroda, that $\varphi_{\lambda}$ is a wild automorphism of $\C^3$, if $\lambda\neq0$ (see e.g. Theorem 2.3. in~\cite{Kuroda-wildness}).

We can now state the main result of our paper as follows.

\begin{theorem}\label{main-thm}
If $n=2m+1$, then
$$\varphi_{\lambda}=\exp(\lambda\delta)=(x-\sum_{k=1}^{m+1}{m+1\choose k}\lambda^k\Delta^ky^{m+1-k}z^{nk-1},y+\lambda\Delta z^n,z)$$
is, for all $\lambda\neq0$,  a wild automorphism of $\C^3$, which belongs to the closure of $\TA_3(\C)$.
\end{theorem}

As a corollary, we thus obtain the following result.

\begin{corollary}
The tame automorphism subgroup is not (weakly) closed in $\GA_3(\C)$.
\end{corollary}

Before proving our main result, let us make two remarks. First, note that the famous Nagata automorphism
$$N=(x-2y(y^2+zx)-z(y^{2}+zx)^2,y+z(y^2+zx),z)$$
corresponds to the case $n=m=1$ (and $\lambda=1$), which does not satisfy the condition $n=2m+1$. Therefore, the following question naturally shows up.

\begin{question}
Does the  Nagata automorphism belong to the closure of the tame automorphism subgroup of $\C^3$?
\end{question}

Let us also point out that every tame automorphism of $\C^3$ can be easily obtained as a limit of wild automorphisms.

\begin{proposition}\label{prop:wilddense}
The set  $\GA_3(\C)\smallsetminus\TA_3(\C)$ of wild automorphisms of $\C^3$ is dense in $\GA_3(\C)$.
\end{proposition}

\begin{proof}Let $\sigma$  be a tame automorphism of $\C^3$ and let $\sigma_t$ be the family of automorphisms of $\C^3$ defined by $\sigma_t:=\varphi_t\circ\sigma$ for all $t\in\C$, where $\varphi_t$ is the exponential map of the locally nilpotent derivation $\delta$ described above. Since $\varphi_0$ is simply equal to the identity map, $\sigma_t$ converges to $\sigma$, when $t\to 0$. Thus, $\sigma$ is a limit of wild automorphisms.
\end{proof}

In order to prove Theorem~\ref{main-thm}, we need to introduce some other notations.  For every  positive integer $m\ge 1$, we let
$$P_m(U)=\sum_{k=0}^m\binom{m+\frac{1}{2}}{k}U^k\in\Q[U].$$
Note that the polynomial $P_m(U)$ is equal to the formal power series of $(1+U)^{m+\frac{1}{2}}$ truncated at the order $m$. We consider also the two following triangular (tame) automorphisms of $\A_{\C(t)}^3$.
$$F_t=(x,y,z+t^{m+1}(tx^2-y^{2m+1}))\quad \text{and}\quad G_t=(x+\frac{z^{2m+1}}{t^{m+1}}P_m(\frac{ty}{z^2}),y+\frac{z^2}{t},z).$$
Remark that $G_t$ indeed defines  an automorphism of $\A_{\C(t)}^3$, since $$\frac{z^{2m+1}}{t^{m+1}}P_m(\frac{ty}{z^2})=\sum_{k=0}^m\binom{m+\frac{1}{2}}{k}t^{-(m+1-k)}y^kz^{2m+1-2k}$$
is an element of $\C(t)[y,z]$.  Finally, we set $\sigma_t=G_t^{-1}\circ F_t\circ G_t\in\GA_3(\C(t))$. It turns out that the map $\sigma_t$ has all its coefficients in $\C[t]$. More precisely, we have the following result.

\begin{theorem}\label{thm-wild-limits}
Let $\sigma_t$ be the tame automorphism of $\A_{\C(t)}^3$ defined above. Then all three components  of $\sigma_t$ are elements of $\C[t][x,y,z]$. Moreover, putting $t=0$ in their formulas, we get a wild polynomial automorphism of $\C^3$, whose last two components are $y-4z^{2m+1}(xz-\binom{m+\frac{1}{2}}{m+1}y^{m+1})$ and $z$, respectively.
\end{theorem}

\begin{proof}
Remark that the inverse $G_t^{-1}$ of  $G_t$ is given by
$$G_t^{-1}=(x-\frac{z^{2m+1}}{t^{m+1}}P_m(\frac{ty-z^2}{z^2}),y-\frac{z^2}{t},z).$$
Note that $z^{2m+1}P_m(\frac{ty-z^2}{z^2})\in\C[t][x,y,z]$. In particular, the components of $\sigma_t=G_t^{-1}\circ F_t\circ G_t$ are all elements of $\C[t,t^{-1}][x,y,z]$. Let us denote them by $X:=\sigma_t^*(x)$, $Y:=\sigma_t^*(y)$ and $Z:=\sigma_t^*(z)$, respectively. By construction, it is clear that $\sigma_t$ is a tame automorphism of $\A_{\C(t)}^3$. We will further prove the following assertions.

\begin{enumerate}
\item $X,Y,Z\in\C[t,x,y,z]$.
\item $Z\equiv z\mod(t)$.
\item $Y\equiv y-4z^{2m+1}(xz-\binom{m+\frac{1}{2}}{m+1}y^{m+1})\mod(t)$.
\item $\widetilde{\sigma}:=(X|_{t=0},Y|_{t=0},Z|_{t=0})$ is a wild automorphism of $\C^3$.
\end{enumerate}

Let us first compute $Z$.
Setting $T=\frac{ty}{z^2}$, we get $y+\frac{z^2}{t}=\frac{z^2}{t}(1+T)$, and thus
\begin{align*}
Z & = z+t^{m+1}\left(t\left(x+\frac{z^{2m+1}}{t^{m+1}}P_m(T)\right)^2-\left(\frac{z^2}{t}(1+T)\right)^{2m+1}\right)\\
  & = z+t^{m+2}x^2+2txz^{2m+1}P_m(T)+\frac{z^{4m+2}}{t^m}\left(P_m(T)^2-(1+T)^{2m+1}\right).
\end{align*}
Since $P_m(U)$ is the formal power series of $(1+U)^{m+\frac{1}{2}}$ truncated at the order $m$, we have $P_m(U)^2\equiv(1+U)^{2m+1}$ modulo $(U^{m+1})$.
So, $P_m(T)^2\equiv(1+T)^{2m+1}$ modulo $(t^{m+1})$ and we obtain therefore that $Z=\sigma_t^*(z)\in\C[t,x,y,z]$ and $Z|_{t=0}=z$. This proves Assertion (2).

Now, $Y=\sigma_t^*(y)=y+\frac{z^2}{t}-\frac{Z^2}{t}= y-(Z+z)\frac{Z-z}{t}$.
Since the formal power series of $(1+U)^{m+\frac{1}{2}}$ truncated at the order $m+1$ is equal to $P_m(U)+{m+\frac{1}{2}\choose m+1}U^{m+1}$, we have  $P_m(U)^2\equiv(1+U)^{2m+1}-2{m+\frac{1}{2}\choose m+1}U^{m+1}$.
From this, we deduce that $P_m(T)^2\equiv(1+T)^{2m+1}-2{m+\frac{1}{2}\choose m+1}T^{m+1}$ modulo $(t^{m+2})$, where $T=\frac{ty}{z^2}$. Thus,
\begin{align*}
Z & \equiv  z+2txz^{2m+1}P_m(T)+\frac{z^{4m+2}}{t^m}\left(-2{m+\frac{1}{2}\choose m+1}T^{m+1}\right) & \mod(t^2)\\
 & \equiv  z+2txz^{2m+1}-2{m+\frac{1}{2}\choose m+1}ty^{m+1}z^{2m} &\mod(t^2).
\end{align*}
Therefore, $\sigma_t^*(y)=y-(Z+z)\frac{Z-z}{t}\in\C[t][x,y,z]$ and
$$Y|_{t=0}=y-4z^{2m+1}(xz-\binom{m+\frac{1}{2}}{m+1}y^{m+1}).$$
This proves Assertion (3). For $X=\sigma_t^*(x)$, we have
\begin{align*}
X & =  x+ \frac{z^{2m+1}}{t^{m+1}}P_m(\frac{ty}{z^2})-  \frac{Z^{2m+1}}{t^{m+1}}P_m(\frac{ty+z^2-Z^2}{Z^2})\\
  & =  x+ \frac{1}{t^{m+1}}\left(z^{2m+1}P_m(\frac{ty}{z^2})- Z^{2m+1}P_m(\frac{tY}{Z^2})\right)\\
  & =  x+ \frac{1}{t^{m+1}}\sum_{k=0}^m\binom{m+\frac{1}{2}}{k}t^{k}\left(y^kz^{2m+1-2k}-Y^kZ^{2m+1-2k}\right)
\end{align*}

We will prove that $z^{2m+1}P_m(\frac{ty}{z^2})- Z^{2m+1}P_m(\frac{tY}{Z^2})$, which is an element of $\C[x,y,z,t]$, is congruent to 0 modulo $(t^{m+1})$. It suffices to prove that
$$\left(z^{2m+1}P_m(\frac{ty}{z^2})\right)^2-\left(Z^{2m+1}P_m(\frac{tY}{Z^2})\right)^2\equiv0\mod(t^{m+1})$$
and that
$$z^{2m+1}P_m(\frac{ty}{z^2})+Z^{2m+1}P_m(\frac{tY}{Z^2})\not\equiv0\mod(t) .$$
The second assertion holds since $z^{2m+1}P_m(\frac{ty}{z^2})\equiv z^{2m+1}$ and $Z^{2m+1}P_m(\frac{tY}{Z^2})\equiv Z^{2m+1}\equiv z^{2m+1}$ modulo $(t)$. For the first one, recall that $P_m(U)^2$ is congruent to $(1+U)^{2m+1}$ modulo $(U^{m+1})$. Therefore, we have the following congruences modulo $(t^{m+1})$.
$$\left(z^{2m+1}P_m(\frac{ty}{z^2})\right)^2\equiv (z^2)^{2m+1}(1+\frac{ty}{z^2})^{2m+1}\equiv (z^2+ty)^{2m+1}$$
and
$$\left(Z^{2m+1}P_m(\frac{tY}{Z^2})\right)^2\equiv  (Z^2+tY)^{2m+1}=(z^2+ty)^{2m+1}.$$
This implies the desired result, and so Assertion (1) follows.

It remains to show Assertion (4). To check that $\widetilde{\sigma}$ is a polynomial automorphism of $\C^3$, we can consider the natural extension of $F_t$ and $G_t$ as birational maps of $\C^4_{t,x,y,z}$ fixing the first coordinate. Note that their Jacobian are both equal to 1. Thus,  the endomorphism $\varphi$ of $\C^4_{t,x,y,z}$ defined by $\varphi=(t,X,Y,Z)$ is also of Jacobian 1 and it is a birational map. This implies (see for example Corollary 1.1.34 in \cite{Essen}) that $\varphi$ is a polynomial automorphism of $\C^4$. In particular, $\widetilde{\sigma}$ is an automorphism of $\C^3$. By Assertion (2), $\widetilde{\sigma}$ is an element of $\GA_2(\C[z])$. Finally, Assertion $(3)$ and the main result of \cite{SU2} allow us to conclude that $\widetilde{\sigma}$ is a wild automorphism of $\C^3$, since $Y|_{t=0}$ is a wild coordinate of $\C[z][x,y]$ (see e.g. Proposition 2 in~\cite{EdoVenereau}).
\end{proof}

\begin{example}
For $m=1$, we obtain that
$$\sigma_t=(x-\frac{3yz}{2t}+\frac{z^3}{2t^2},y-\frac{z^2}{t},z)\circ(x,y,z+t^3x^2-t^2y^{3})\circ(x+\frac{3yz}{2t}+\frac{z^3}{t^2},y+\frac{z^2}{t},z)$$
is a tame automorphism of $\A^3_{\C(t)}$ which has all its coefficient in $\C[t]$. Computing explicitly these coefficients and letting $t=0$ in the formulas, we find then the following wild automorphism of $\C^3$, which is the limit when $t\to0$ of the family $(\sigma_t)_{t\neq0}$ of tame automorphisms of $\C^3$.
\begin{align*}\widetilde{\sigma} &= (x+\frac{9}{8}y^3z^2-3xyz^3+\frac{27}{32}y^4z^5-\frac{9}{2}xy^2z^6+6x^2z^7, y+\frac{3}{2}y^2z^3-4xz^4, z)\\
&=\left(x+\frac{3}{4}z^2y(\frac{3}{2}y^2-4xz)+\frac{3}{8}z^5(\frac{3}{2}y^2-4xz)^2, y+z^3(\frac{3}{2}y^2-4xz), z\right).
\end{align*}
\end{example}

Finally, we prove Theorem~\ref{main-thm}.

\begin{proof}(of Theorem~\ref{main-thm}). Let $m\geq1$ be a fixed integer.
For every $\lambda\in\C^*$, let $\Psi_{\lambda}$ be the affine automorphism of $\C^3$ defined by $\Psi_{\lambda}=(ax,by,cz)$, where $a,b,c\in\C^*$ are chosen such that $-\binom{m+\frac{1}{2}}{m+1}b^{m+1}=1$, $-4c^{2m+1}=\lambda b$ and $ac=1$. Then, consider the  automorphism  $\alpha_{t}$ of $\A_{\C(t)}^3$ given by $\alpha_{t}=\Psi_{\lambda}^{-1}\circ\sigma_t\circ\Psi_{\lambda}$, where  $\sigma_{t}$  denotes the automorphism defined before Theorem~\ref{thm-wild-limits}. By Theorem~\ref{thm-wild-limits}, we have that $(\alpha_t)_{t\in\C^*}$ is a family of tame automorphisms of $\C^3$, which converges to a wild automorphism $\alpha$ of $\C^3$, when $t\to0$. Moreover, the two last components of $\alpha$ are equal to $y+\lambda(xz+y^{m+1}) z^{2m+1}$ and $z$, respectively. Note that these two  are also the last components of $\varphi_{\lambda}=\exp(\lambda\delta)$ in the case $n=2m+1$. Therefore, there exists a tame  polynomial automorphism $f$ of $\C^3$ of the form $f=(dx+P(y,z),y,z)$ with $d\in\C^*$ and $P(y,z)\in\C[y,z]$ such that $f\circ\alpha=\varphi_{\lambda}$. Thus, $(f\circ\alpha_t)_{t\neq0}$ is a family of tame automorphisms, which converges to $\varphi_{\lambda}$ for $t\to0$. This proves the theorem.
\end{proof}

\section{Acknowledgements}
We are very grateful to Jean-Philippe Furter and Hanspeter Kraft  for sharing with us preliminary version of their upcoming paper about the geometry of the automorphism group of affine $n$-space, and for fruitful discussions, which helped us to understand some subtleties of the topology of ind-groups.


\begin{bibdiv}
\begin{biblist}

\end{biblist}
\end{bibdiv}


\begin{bibdiv}
\begin{biblist}

\bib{BCW}{article}{
   author={Bass, Hyman},
   author={Connell, Edwin H.},
   author={Wright, David},
   title={The Jacobian conjecture: reduction of degree and formal expansion
   of the inverse},
   journal={Bull. Amer. Math. Soc. (N.S.)},
   volume={7},
   date={1982},
   number={2},
   pages={287--330},
}

\bib{EdoVenereau}{article}{
   author={Edo, Eric},
   author={V{\'e}n{\'e}reau, St{\'e}phane},
   title={Length 2 variables of $A[x,y]$ and transfer},
   note={Polynomial automorphisms and related topics (Krak\'ow, 1999)},
   journal={Ann. Polon. Math.},
   volume={76},
   date={2001},
   number={1-2},
   pages={67--76},
   issn={0066-2216},
}
\bib{Essen}{book}{
   author={van den Essen, Arno},
   title={Polynomial automorphisms and the Jacobian conjecture},
   series={Progress in Mathematics},
   volume={190},
   publisher={Birkh\"auser Verlag},
   place={Basel},
   date={2000},
   pages={xviii+329},
   isbn={3-7643-6350-9},
}

\bib{Kulk}{article}{
   author={van der Kulk, W.},
   title={On polynomial rings in two variables},
   journal={Nieuw Arch. Wiskunde (3)},
   volume={1},
   date={1953},
   pages={33--41},
}


\bib{JPF}{article}{
   author={Furter, Jean-Philippe},
   title={On the variety of automorphisms of the affine plane},
   journal={J. Algebra},
   volume={195},
   date={1997},
   number={2},
   pages={604--623},
}

\bib{FurterLength}{article}{
   author={Furter, Jean-Philippe},
   title={On the length of polynomial automorphisms of the affine plane},
   journal={Math. Ann.},
   volume={322},
   date={2002},
   number={2},
   pages={401--411},
}

\bib{Furter}{article}{
   author={Furter, Jean-Philippe},
   title={Plane polynomial automorphisms of fixed multidegree},
   journal={Math. Ann.},
   volume={343},
   date={2009},
   number={4},
   pages={901--920},
}

\bib{FK}{article}{
   author={Furter, Jean-Philippe},
   author={Kraft, Hanspeter},
   title={On the geometry of the automorphism group of affine n-space},
   journal={in preparation},
   date={2014},
}

\bib{humphreys}{book}{
   author={Humphreys, James E.},
   title={Linear algebraic groups},
   note={Graduate Texts in Mathematics, No. 21},
   publisher={Springer-Verlag},
   place={New York},
   date={1975},
   pages={xiv+247},
 }

\bib{Jung}{article}{
   author={Jung, Heinrich W. E.},
   title={\"Uber ganze birationale Transformationen der Ebene},
   journal={J. Reine Angew. Math.},
   volume={184},
   date={1942},
   pages={161--174},
   issn={0075-4102},
}
		
\bib{Kuroda}{article}{
   author={Kuroda, Shigeru},
   title={Shestakov-Umirbaev reductions and Nagata's conjecture on a
   polynomial automorphism},
   journal={Tohoku Math. J. (2)},
   volume={62},
   date={2010},
   number={1},
   pages={75--115},
   issn={0040-8735},
}
	
\bib{Kuroda-wildness}{article}{
   author={Kuroda, Shigeru},
   title={Wildness of polynomial automorphisms in three variables},
   journal={preprint arXiv:1110.1466 [math.AC]},
 date={2011},
}


\bib{Sha66}{article}{
   author={Shafarevich, Igor R.},
   title={On some infinite-dimensional groups},
   journal={Rend. Mat. e Appl. (5)},
   volume={25},
   date={1966},
   number={1-2},
   pages={208--212},
}

\bib{Sha81}{article}{
   author={Shafarevich, Igor R.},
   title={On some infinite-dimensional groups. II},
   language={Russian},
   journal={Izv. Akad. Nauk SSSR Ser. Mat.},
   volume={45},
   date={1981},
   number={1},
   pages={214--226, 240},
   issn={0373-2436},
}	



\bib{Sha95}{article}{
   author={Shafarevich, Igor R.},
   title={Letter to the editors: ``On some infinite-dimensional groups. II''},
   language={Russian},
   journal={Izv. Ross. Akad. Nauk Ser. Mat.},
   volume={59},
   date={1995},
   number={3},
   pages={224},
   issn={0373-2436},
}

\bib{SU1}{article}{
   author={Shestakov, Ivan P.},
   author={Umirbaev, Ualbai U.},
   title={Poisson brackets and two-generated subalgebras of rings of
   polynomials},
   journal={J. Amer. Math. Soc.},
   volume={17},
   date={2004},
   number={1},
   pages={181--196},
   issn={0894-0347},
}

\bib{SU2}{article}{
   author={Shestakov, Ivan P.},
   author={Umirbaev, Ualbai U.},
   title={The tame and the wild automorphisms of polynomial rings in three
   variables},
   journal={J. Amer. Math. Soc.},
   volume={17},
   date={2004},
   number={1},
   pages={197--227},
   issn={0894-0347},
}


\end{biblist}
\end{bibdiv}
\end{document}